\documentclass{article}

\usepackage{setspace}
\usepackage{algorithm} 
\usepackage{algpseudocode}
\usepackage{svg}

\usepackage{multicol}

\usepackage{conf/deps}
\usepackage{conf/custom}
\usepackage{conf/tikz-cd}

\newcommand{\lar}{\leftarrow}

\newcommand{\rar}{\rightarrow}

\newcommand{\incl}{\subseteq}

\newcommand{\cont}{\supseteq}

\newcommand{\imply}\Rightarrow
\newcommand{\eqval}\Leftrightarrow
\newcommand{\txt}[1]{\mathrm{\quad #1 \quad}}
\renewcommand{\and}{\txt{and}}

\newcommand{\st}{\; | \;}

\newcommand\braket[2]{\langle\, #1 \,| \,  #2\, \rangle}

\DeclareMathOperator{\e}{e}

\DeclareMathOperator{\Img}{Im}
\DeclareMathOperator{\Ker}{Ker}

\newcommand{\Part}{\mathcal{P}}

\newcommand{\R}{\mathbb{R}}
\newcommand{\Z}{\mathbb{Z}}

\newcommand{\Set}{\mathbf{Set}}

\newcommand{\Top}{\mathbf{Top}}
\newcommand{\Prob}{\mathrm{Prob}}

\DeclareFontFamily{OT1}{pzc}{}
\DeclareFontShape{OT1}{pzc}{m}{it}{ <-> s * [1.1] pzcmi7t}{}
\DeclareMathAlphabet{\mathpzc}{OT1}{pzc}{m}{it}
{

\newcommand{\ph}{\varphi}
\renewcommand{\aa}{\alpha}
\newcommand{\bb}{\beta}
\newcommand{\cc}{\gamma}

\renewcommand{\ee}{\varepsilon}

\newcommand{\Om}{\Omega}
\newcommand{\LL}{\Lambda}

\newcommand{\DF}{\mathpzc{D}}

\newcommand{\Fh}{\mathbb{F}}

\begin{document}

\title{Belief Propagation as Diffusion}
\author{Olivier Peltre\\[0.4em]
olivier.peltre@univ-artois.fr\\[0.4em]
Université d'Artois, Faculté Jean Perrin (LML)\\[0.4em]
Rue Jean Souvraz 62307 LENS CEDEX\\[0.4em]
}
\date{2021}

\maketitle

\begin{abstract} 
  We introduce novel belief propagation algorithms to estimate 
  the marginals of a high dimensional probability distribution. 
  They involve natural (co)homological constructions relevant for 
  a localised description of statistical systems.
\end{abstract} 

\section*{Introduction}

Message-passing algorithms such as belief propagation (BP) are
parallel computing schemes that try to estimate the marginals of a 
high dimensional probability distribution. 
They are used in various areas involving the statistics 
of a large number of interacting random variables,
such as computational thermodynamics \cite{Kikuchi-51,Mezard-Montanari}, 
artificial intelligence \cite{Pearl-82,Yedidia-2005,Ping-2017}, computer vision \cite{Sun-2003}
and communications processing \cite{Gallager-63,Jego-2009}. 

We have shown the existence of a non-linear 
correspondence between BP algorithms 
and discrete integrators of a new form of continuous-time 
diffusion equations on belief networks \cite{gsi19,phd}. 
Practical contributions include
(a) regularised BP algorithms for any time step 
or {\it diffusivity}\footnote{
  This coefficient $\ee$ would appear as an exponent of messages 
  in the usual multiplicative writing of BP equations. 
  Diffusivity relates energy density gradients to heat fluxes in physics, 
  as in $\vec {\ph} = - \ee \cdot \vec\nabla(u)$. 
}
coefficient $0 < \ee < 1$, 
and (b) a canonical {\it Bethe diffusion flux} 
that regularises GBP messages
by new Möbius inversion formulas in degree 1\footnote{
  Generalised belief propagation = BP on hypergraphs, 
  see \cite{Yedidia-2005} for the algorithm. Our algorithm 2
  exponentiates their messages $m_{\aa\bb}$ 
  by the coefficients $c_\aa \in \Z$ appearing in the 
  Bethe-Kikuchi local approximation of free energy. 
}.

The purpose of this text is to describe  
the structure of belief networks as concisely as possible, 
with the geometric operations that appear in 
our rewriting of BP equations. An open-sourced python implementation,
hosted on github at \href{https://github.com/opeltre/topos}{opeltre/topos}
was also used to conduct benchmarks 
showing the importance of chosing $\ee < 1$. 

In the following, we denote by: 
\begin{itemize}
\iii $\Om = \{i, j, k, \dots \}$ a finite set of indices
(e.g. atoms, neurons, pixels, bits ...)
\iii $x_i$ the microstate of atom $i$, valued in a finite set $E_i$
\iii $x_\Om$ the microstate of the global system, valued in 
$E_\Om = \prod_{i \in \Om} E_i$ 
\end{itemize}
The statistical state of the system is described by a 
probability distribution $p_\Om$ on $E_\Om$.
We write $\Delta_\Om = \Prob(E_\Om)$ for the convex space
of statistical states. 

\section{Graphical Models} 

\begin{definition} 
A {\em hypergraph} $(\Om, K)$ is a set of {\em vertices} 
$\Om$ and a set of {\em faces}\footnote{
  Also called hyperedges, or regions. 
  A {\it graph} is a hypergraph with only hyperedges of cardinality 2.
  A {\it simplicial complex} 
  is a hypergraph such that any subset of a face is also a face. 
  A {\it lattice} is a hypergraph closed 
  under $\cap$ and $\cup$. 
  We shall mostly 
  be interested in {\it semi-lattices},
  closed only under intersection, of which simplicial complexes are 
  a special case. 
}
$K \incl \Part(\Om)$. 
\end{definition} 

Let us denote by $x_\aa$ the microstate of a face $\aa \incl \Om$, 
valued in $E_\aa = \prod_{i \in \aa} E_\aa$. For 
every $\bb \incl \aa$ in $\Part(\Om)$, we have a canonical projection or 
{\it restriction}\footnote{
  The contravariant 
  functor $E : \Part(\Om)^{op} \to \Set$ of microstates defines 
  a sheaf of sets over $\Om$. 
} map:
\[ \pi^{\bb\aa} : E_\aa \to E_\bb \] 
We simply write $x_\bb$ 
for the restriction of $x_\aa$ 
to a subface $\bb$ of $\aa$. 

\begin{definition} \label{def-GM}
A {\em graphical model} $p_\Om \in \Delta_\Om$
on the hypergraph $(\Om, K)$ is a 
positive probability distribution on $E_\Om$ that factorises 
as a product of positive local factors over faces: 
\[ p_\Om(x_\Om) = \frac 1 {Z_\Om} \prod_{\aa \in K} f_\aa(x_\aa) 
  = \frac 1 {Z_\Om} \e^{- \sum_\aa h_\aa(x_\aa)}  \]
We denote by $\Delta_K \incl \Delta_\Om$ the subspace 
of graphical models on $(\Om, K)$.
\end{definition}

\begin{figure}[b]
\begin{center}
\begin{minipage}{0.8\textwidth}
  \begin{center}
  \img{0.9}{GM.pdf}
  \vspace{0.4em}
  \end{center}
  \small
  {\bf Fig 1.} 
    Graphical model 
    $p_{ijkl}(x_{ijkl}) 
    = f_{ijk}(x_{ijk}) \cdot f_{ikl}(x_{ikl}) \cdot f_{jkl}(x_{jkl})$
    with its factor graph representation (middle)  
    on a simplicial complex $K$ formed by joining 3 triangles at 
    a common vertex and called 2-{\it horn} $\LL^2$ 
    of the 3-simplex (left). 
    The situation is equivalent when $K$ is a three-fold 
    covering of $\Om$ by 
    intersecting regions $\aa, \aa', \aa''$ (right).
\end{minipage}
\end{center}
\end{figure}

A graphical model $p_\Om$ for $(\Om, K)$ is also called
{\it Gibbs state} of the associated energy function or
hamiltonian $H_\Om : E_\Om \to \R$: 
\[ H_\Om(x_\Om) = \sum_{\aa \in K} h_\aa(x_\aa) \]
The normalisation factor of the Gibbs density  $\e^{-H_\Om}$ 
is computed by the partition function 
$Z_\Om 
= \sum_{x_\Om} \e^{- H_\Om(x_\Om)}$. 
The free energy 
$F_\Om = - \ln Z_\Om$ and partition function 
generate most relevant statistical quantities in their 
derivatives\footnote{
  Letting $\mu_H$ denote the image by $H$ 
  of the counting measure on microstates,  
  $Z^\theta_\Om 
  = \int_{\lambda \in \R} \e^{- \theta \lambda} \mu_H(d\lambda)$ 
  is the Laplace transform 
  of $\mu_H$ with respect to inverse temperature 
  $\theta = 1 / {k_B T}$. 
  In \cite{phd} we more generally consider free energy as 
  a functional $A_\Om \to \R$ whose differential at $H_\Om \in A_\Om$ 
  is the Gibbs state $p_\Om \in A_\Om^*$. 
}. 
They are however not computable in practice, the sum over 
microstates scaling exponentially in the number of atoms. 

Message-passing algorithms rely on local structures 
induced by $K$ to estimate marginals, providing with an 
efficient alternative \cite{Kikuchi-51,Mezard-Montanari} 
to Markov Chain Monte Carlo methods such as Hinton's contrastive
divergence algorithm commonly used for training 
restricted Boltzmann machines \cite{Hinton-2009,Ping-2017}. 
They are also related to local variational principles 
involved in the estimation of
$F_\Om$ \cite{Yedidia-2005,gsi19,phd} 
by Bethe approximation \cite{Bethe-35,Kikuchi-51}. 

We showed in \cite{phd} that message-passing explores
a subspace of potentials $(u_\aa)$ 
related to equivalent factorisations of $p_\Om$,  until an associated 
collection of local probabilities $(q_\aa)$ is consistent. 
Two fundamental operations constraining this  
non-linear correspondence are introduced below. 
They consist of a differential $d$ 
associated to a {\it consistency} constraint, 
and its adjoint boundary $\delta = d^*$ enforcing  
a dual {\it energy conservation} constraint. 
These operators relate graphical models to a 
statistical (co)homology theory, in addition to 
generating the BP equations.

\section{Marginal Consistency}

In the following, we suppose given a hypergraph $(\Om, K)$ 
closed under intersection: 
\[ \aa \cap \bb \in K \quad\txt{for\;all}\quad  
 \aa, \bb \in K \]
We denote by $\Delta_\aa$
the space of probability distributions on $E_\aa$ for all $\aa \in K$. 
Given a graphical model $p_\Om \in \Delta_\Om$, 
the purpose of belief propagation algorithms is 
to efficiently approximate the collection of true marginals 
$p_\aa \in \Delta_\aa$ for $\aa \in K$ 
by local beliefs $q_\aa \in \Delta_\aa$, 
in a space $\Delta_0$ of dimension typically much smaller than $\Delta_\Om$. 

\begin{definition} 
We call {\em belief} over $(\Om, K)$ a collection $q \in \Delta_0$
of local probabilities over faces, where: 
  \[ \Delta_0 = \prod_{\aa \in K} \Delta_\aa \] 
\end{definition} 

\begin{definition} 
For every $\bb \incl \aa$ the {\em marginal} or 
partial integration map 
$\Sigma^{\bb\aa} : \Delta_\aa \to \Delta_\bb$ is defined by: 
  \[ \Sigma^{\bb\aa} q_\aa (x_\bb) 
  = \sum_{y \in E_{\aa\setminus\bb}} q_\aa(x_\bb, y) 
  \]
\end{definition}

\begin{definition}
{\em Consistent} beliefs span the convex subset 
$\Gamma \incl \Delta_0$ defined by marginal consistency 
constraints\footnote{
  Equivalently, $\Gamma$ is the projective limit 
  of the functor $\Delta : K^{op} \to \Top$ defined by local 
  probabilities and marginal projections, or 
  space of global sections of the sheaf
  of topological spaces $\Delta$ over $(\Om, K)$.  
}: 
\[ q_\bb = \Sigma^{\bb\aa}(q_\aa) \quad\txt{for\;all}\; \bb \incl \aa \]
\end{definition}

The true marginals $(p_\aa) \in \Delta_0$ of 
a global density $p_\Om \in \Delta_\Om$ 
are {\it always consistent}. 
However their symbolic definition
$p_\aa = \Sigma^{\aa\Om} p_\Om$ 
involves a sum over fibers of $E_{\Om \setminus \aa}$,
not tractable 
in practice. 
Message-passing algorithms instead explore a parameterised family
of beliefs $q \in \Delta_0$ until meeting the consistency constraint 
surface $\Gamma \incl \Delta_0$.

Let us denote by $A_\aa^*$ the space of linear measures on $E_\aa$ 
for all $\aa \incl \Om$, and by: 
\[ \Sigma^{\bb\aa} : A^*_\aa \to A^*_\bb \]
the partial integration map.

\begin{definition} 
  We call $n$-{\em density} over $(\Om, K)$ an element 
  $\lambda \in A_n^*$ of local measures indexed by ordered 
  chains of faces, where:
  \[ A_n^* = \prod_{\aa_0 \supset \dots \supset \aa_n} A_{\aa_n}^* \]
\end{definition}

The marginal consistency constraints are expressed by a  
differential operator\footnote{
  Cohomology sequences of this kind were considered by Grothendieck and 
  Verdier \cite{SGA-4-V}, see also \cite{Moerdijk}. 
} 
$d$ on the graded vector space
$A_\bullet^* = \prod_n A_n^*$ of densities over $(\Om, K)$:
\[ \bcd 
A_0^* \rar{d} & A_1^* \rar{d} & \dots \rar{d} & A_n^* 
\ecd \]

\begin{definition} 
  The {\em differential} $d : A_0^* \to A_1^*$ acts on a density 
  $(\lambda_\aa) \in A_0^*$ by: 
  \[ d(\lambda)_{\aa\bb} = \lambda_\bb - \Sigma^{\bb\aa} \lambda_\aa \]
  {\em Consistent} densities $\lambda \in [A_0^*]$ 
  satisfy $d \lambda = 0$, and called $0$-cocycles. 
\end{definition}

The space of consistent beliefs $\Gamma \incl [A_0^*]$ 
is the intersection of $\Ker(d)$ 
with $\Delta_0 \incl A_0^*$. 
True marginals define a convex map $\Delta_\Om \to \Gamma$,
restriction\footnote{
  Note the image of $\Delta_\Om$ inside $\Gamma$ can be 
  a strict convex polytope of $\Gamma$, 
  and consistent $q \in \Gamma$ 
  do not always admit a positive preimage $q_\Om \in \Delta_\Om$ 
  \cite{Vorobev-62,Abramsky-2011}. 
}
of a linear surjection $A^*_\Om \to [A_0^*]$.
Consistent beliefs $q \in \Gamma$ acting as 
for global distributions $p_\Om \in \Delta_\Om$, 
marginal diffusion iterates over a smooth
subspace of $\Delta_0$, diffeomorphic to equivalent parameterisations 
of a graphical model $p_\Om$, until eventually reaching $\Gamma$. 

\section{Energy Conservation} 

Graphical models parameterise 
a low dimensional subspace of $\Delta_\Om$, 
but definition \ref{def-GM} is not 
injective in the local factors $f_\aa$ or local potentials 
$u_\aa = - \ln f_\aa$. 
The fibers of this parameterisation can be described linearly 
at the level of potentials, 
and correspond to homology classes of the codifferential operator 
$\delta = d^*$. 

We denote by $A_\aa$ the algebra of real functions on $E_\aa$ 
for all $\aa \incl \Om$, and by: 
\[ j_{\aa\bb} : A_\aa \to A_\bb \]
the natural extension\footnote{
  Functions on $E_\bb = \prod_{j \in \bb} E_j$ 
  can be viewed as functions on $E_\aa = \prod_{i \in \aa}$ 
  that do not depend on the state of $x_i$ for 
  $i \in \aa \setminus \bb$. Therefore $A_\bb$ 
  is essentially a subspace of $A_\aa$ 
  and $j_{\aa\bb}$ an inclusion.
} of functions pulled from $E_\bb$ to $E_\aa$ by
the restriction $x_\aa \mapsto x_\bb$. 

\begin{definition} 
  We let $\delta = d^*$ denote the adjoint of $d$, defined by duality:
  \[ \bcd
  A_0 & A_1 \lar[swap]{\delta} & \dots \lar[swap]{\delta} &
  A_n \lar[swap]{\delta} \ecd \]
\end{definition}

\begin{proposition} The {\em divergence} $\delta: A_1 \to A_0$ 
dual of $d : A_0^* \to A_1^*$, acts on $\ph \in A_1$ by:
  \[ \delta(\ph)_\bb = \sum_{\aa \cont \bb} \ph_{\aa\bb} 
  - \sum_{\cc \incl \bb} j_{\bb\cc} \ph_{\bb\cc} 
  \] 
\end{proposition} 

\begin{proof} 
  Let $\lambda \in A_0^*$ and $\ph \in A_1$.
  The duality bracket $A_0^* \otimes A_0 \to \R$ 
  is naturally defined by sum 
  of local duality brackets $A_\bb^* \otimes A_\bb \to \R$, 
  which correspond to integration of local measures against observables: 
  \[ \braket{\lambda}{\delta \ph} = 
  \sum_{\bb \in K} \braket{\lambda_\bb}{\delta \ph_\bb} 
  = \sum_{\bb \in K} \sum_{x_\bb \in E_\bb} 
  \lambda_\bb(x_\bb) \delta\ph_\bb(x_\bb) \]
  Substituting with the expression of $\delta \ph$ we get\footnote{
    In this substitution, we simply wrote $\ph_{\bb\cc}(x_\cc)$ for 
    $j_{\bb\cc}(\ph_{\bb\cc})(x_\bb)$, as $j_{\bb\cc} : A_\cc \to A_\bb$ 
    is an inclusion. 
  }: 
  \[ \begin{split}
    \braket{\lambda}{\delta \ph} 
  &= \sum_{\bb \in K} \: \sum_{x_\bb \in E_\bb} 
    \lambda_\bb(x_\bb) \Big( 
    \sum_{\aa \cont \bb} \ph _{\aa\bb}(x_\bb) - 
    \sum_{\cc \incl \bb} \ph_{\bb\cc}(x_{\cc}) \Big) \\[0.8em]
  &= \sum_{\aa \cont \bb} \: \sum_{x_\bb \in E_\bb} 
    \ph_{\aa\bb}(x_\bb) \lambda_\bb(x_\bb)
    - \sum_{\bb \cont \cc} \: \sum_{x_\cc \in E_\cc} 
    \ph_{\bb\cc}(x_\cc) 
    \sum_{y \in E_{\bb \setminus \cc}} \lambda_\bb(x_\cc, y) 
  \end{split} \] 
  The factorisation of the rightmost sum by $\ph_{\bb\cc}(x_\cc)$ reflects 
  the duality of $\Sigma^{\bb\aa}$ with $j_{\bb\cc}$. 
  Relabeling summation indices $\bb \cont \cc$ as $\aa \cont \bb$, 
  we finally get: 
  \[ \sum_{\aa \cont \bb} \braket{\lambda_\bb}{\ph_{\aa\bb}} 
    - \sum_{\bb \cont \cc} 
      \braket{\Sigma^{\cc\bb} \lambda_\bb}{\ph_{\bb\cc}} 
    = \sum_{\aa \cont \bb} 
    \braket{\lambda_\bb - \Sigma^{\bb\aa}\lambda_\aa}{\ph_{\aa\bb}} \\
  \]
  So that $\braket{\lambda}{\delta \ph} = \braket{d\lambda}{\ph}$ 
  for all $\lambda \in A_0^*$ and all $\ph \in A_1$. 
\end{proof} 

  Consider the total energy map 
  $\zeta_\Om : A_0 \to A_\Om$ defined by: 
  \[ \zeta_\Om(u) = \sum_{\aa \in K} u_\aa \]
  We have left injections $j_{\Om\aa}$ implicit, 
  viewing each $A_\aa \incl A_\Om$ as a subalgebra of $A_\Om$.
  Denoting by $A_K \incl A_\Om$ the image of $\zeta_\Om$, 
  a graphical model $p_\Om \in \Delta_K$ 
  is then associated to $u \in A_0$ by normalising the 
  Gibbs density $\e^{-\zeta_\Om(u)}$, as in \ref{def-GM}.

  \begin{theorem} \label{conservation}
    For all $u, u' \in A_0$ the following are equivalent 
    \cite[Chapter~5]{phd}: 
    \bi
    \iii conservation of total energy 
    $\sum_\aa u'_\aa = \sum_\aa u_\aa$ in $A_\Om$,
    \iii there exists $\ph \in A_1$ 
    such that $u' = u + \delta \ph$ in $A_0$.
    \ei 
  \end{theorem} 
  
  Theorem \ref{conservation} states that $\Ker(\zeta_\Om)$ coincides
  with the image of the divergence $\delta A_1 \incl A_0$.
  The subspace of total energies 
  $\Img(\zeta_\Om) \simeq A_0 / \Ker(\zeta_\Om)$ 
  is therefore isomorphic to 
  the quotient $[A_0] = A_0 / \delta A_1$, 
  formed by homology classes of potentials 
  $[u] = u + \delta A_1 \incl A_0$.
  Global observables 
  of $A_K \incl A_\Om$ can thus be represented 
  by equivalence classes of local potentials in $[A_0]$, 
  homology under $\delta$ giving a local characterisation 
  for the fibers of $\zeta_\Om$.  

\section{Diffusions}

The local approach to the marginal estimation problem,
given $p_\Om = \frac 1 {Z_\Om} \e^{-H_\Om}$, 
consists of using a low dimensional map 
$A_0 \to \Delta_0$ as substitute for the high dimensional parameterisation 
$A_\Om \to \Delta_\Om$, until parameters $u \in A_0$ 
define a consistent belief $q \in \Gamma$ whose 
components $q_\aa \in \Delta_\aa$ estimate the 
true marginals $p_\aa$ of $p_\Om$. 

\[ \bcd  
\Delta_\Om \rar & \Gamma \rar[hook] & \Delta_0 \\ 
A_\Om \uar
  & \left[ A_0 \right] \uar[swap, dashed] \lar[hook] 
  & A_0 \lar[two heads] \uar[swap]
\ecd \] 

Assume the hamiltonian is defined by $H_\Om = \sum_\aa h_\aa$ 
for given $h \in A_0$. According to theorem \ref{conservation}, 
parameters $u \in A_0$ 
will define the same total energy if and only if: 
\[ u = h + \delta \ph \] 
for some heat flux $\ph \in \delta A_1$. The energy conservation 
constraint $[u] = [h]$ therefore restricts parameters
to fibers of the bottom-right arrow in the above diagram. 
The rightmost arrow $A_0 \to \Delta_0$ is given 
by the equations: 
\begin{equation} \label{eq-beliefs}
q_\aa = \frac 1 {Z_\aa} \e^{-U_\aa}
\quad\txt{where}\quad U_\aa = \sum_{\bb \incl \aa} u_\bb 
\end{equation}
The image of $[h]$ in $\Delta_0$ is a smooth non-linear manifold of 
$\Delta_0 \incl A_0^*$, which may intersect the convex polytope 
$\Gamma = \Ker(d) \cap \Delta_0$ of consistent beliefs 
an unknown number of times. Such consistent beliefs 
in $\Gamma \incl \Delta_0$
are the fixed points of belief propagation algorithms. 
The central dashed vertical arrow therefore represents what they 
try to compute, although no privileged $q \in \Gamma$ may be defined from 
$[h] \in A_0$ in general. 

\begin{definition} Given a {\em flux functional} 
  $\Phi : A_0 \to A_1$, we call 
{\em diffusion} associated to $\Phi$
the vector field $\delta \Phi$ on 
$A_0$ defined by: 
\begin{equation} \label{eq-diffusion}
  \frac{du}{dt} = \delta \Phi(u) 
\end{equation} 
Letting $q \in \Delta_0$ be defined by (\ref{eq-beliefs}), we say that 
$\Phi$ is {\em consistent} if $q \in \Gamma \imply \Phi(u) = 0$, and that
$\Phi$ is {\em faithful} if it is consistent
and $\Phi(u) = 0 \imply q \in \Gamma$. 
\end{definition} 

Consistent flux functionals $\Phi$ are constructed 
by composition with two remarkable operators $\zeta : A_0 \to A_0$, 
mapping potentials to local hamiltonians $u \mapsto U$, 
and $\DF : A_0 \to A_1$,  
a non-linear analog of the differential $d : A_0^* \to A_1^*$, 
measuring inconsistency of the local beliefs defined by 
$U \mapsto q$ in (\ref{eq-beliefs}). 
The definition of $\DF$ involves a conditional form 
of free energy $\Fh^{\bb\aa}: A_\aa \to A_\bb$, 
which generates conditional expectation maps
with respect to local beliefs by differentiation\footnote{
  The tangent map of $\DF$ 
  in turn yields differential operators 
  $\nabla_q : A_0 \to A_1 \to \dots$ for all $q \in \Gamma$,
  whose kernels characterise tangent fibers ${\rm T}_q \Gamma$ 
  pulled by the non-linear parameterisation (\ref{eq-beliefs}), 
  see \cite[Chapter~6]{phd}.
}.

\begin{definition} We call {\em effective energy} 
  the smooth map $\Fh^{\bb\aa} : A_\aa \to A_\bb$ defined by: 
  \[ \Fh^{\bb\aa}(U_\aa \st x_\bb) = 
  - \ln \sum_{y \in E_{\aa \setminus \bb}} \e^{-U_\aa(x_\bb, y)} 
  \]
  and {\em effective energy gradient} the smooth map 
  $\DF : A_0 \to A_1$ defined by: 
  \[ \DF(U)_{\aa\bb}
  = U_\bb - \Fh^{\bb\aa}(U_\aa) \]
\end{definition} 

Letting $q = \e^{-U}$ denote local Gibbs densities, 
note that $q \in \Gamma \eqval \DF(U) = 0$ by:
\[ \DF(U)_{\aa\bb} =
\ln \bigg[\: \frac {\Sigma^{\bb\aa} q_\aa} {q_\bb} \:\bigg]
\]
The map $u \mapsto U$ 
is a fundamental automorphism $\zeta$ of $A_0$, 
inherited from the partial order structure of $K$. 
Möbius inversion formulas define its inverse $\mu = \zeta^{-1}$
\cite{Rota-64,Leinster-08,phd}. We have extended $\zeta$ and $\mu$ 
to automorphisms on the full complex $A_\bullet$ in 
\cite[Chapter~3]{phd}, in particular, $\zeta$ and $\mu$  
also act naturally on $A_1$. 

\begin{definition} The {\em zeta} transform
  $\zeta : A_0 \to A_0$ is defined by: 
  \[ \zeta(u)_\aa = \sum_{\bb \incl \aa} u_\bb \]
\end{definition} 

The flux functional $\Phi = - \DF \circ \zeta$ is consistent 
and faithful \cite{phd}, meaning that $\delta \Phi$
is stationary on $u \in A_0$ if and only if associated beliefs 
$q \in \Delta_0$ are consistent. This flux functional
yields the GBP equations of algorithm A 
(up to the normalisation step of line 3, 
ensuring normalisation of beliefs). 
It may however not be optimal. 

We propose another flux functional $\phi = - \mu \circ \DF \circ \zeta$ 
by degree-1 Möbius inversion on heat fluxes in algorithm B. 
It is remarkable that the associated diffusion $\delta \phi$ 
involves only the coefficients $c_\aa \in \Z$ 
originally used by Bethe \cite{Bethe-35} 
to estimate the free energy 
of statistical systems close to their critical temperature. 
These coefficients also appear in the 
cluster variational problem 
\cite{Kikuchi-51,Morita-57,Pelizzola-2005}
on free energy, solved by fixed points of 
belief propagation and diffusion algorithms \cite{phd,Yedidia-2005}. 

It remains open whether fixed points of Bethe diffusion 
are always consistent. 
We were only able to prove this in a neighbourhood of 
the consistent manifold, a property we called local faithfulness 
of the Bethe diffusion flux $\phi$, see \cite[Chapter~5]{phd}. 
Faithfulness proofs are non-trivial and
we conjecture the global faithfulness of $\phi$. 

\begin{definition} 
  The {\em Bethe numbers} $(c_\aa) \in \Z^K$
  are uniquely defined by the equations:
  \[ \sum_{\aa \cont \bb} c_\aa = 1 
  \txt{for\:\;all} \bb \in K \]
\end{definition} 

  \newcommand{\hr}[1]{\noindent\rule{\textwidth}{#1}}

\begin{figure}[t]
\hr{0.4mm} 

  \noindent {\small {\bf Algorithms.} 
GBP and Bethe Diffusions\footnotemark
  .}\\[-0.5em]
\hrule 
\vspace{-0.5em}
\begin{multicols}{2}
  \begin{figure}[H] \begin{minipage}{0.35\textwidth}
    \smaller
    \begin{tabular*}{\linewidth}
    {@{\extracolsep{\fill}} ll}

    {\bf Input:} 
    & potential ${\tt u} \in A_0$ \\ 
    & diffusivity $\ee > 0$  \\
    & number of iterations $\tt n_{it}$ \\

    \end{tabular*}
  \end{minipage}\end{figure} 

  \begin{figure}[H] \begin{minipage}{0.45\textwidth}
    \smaller
    \begin{tabular*}{\linewidth}
    {@{\extracolsep{\fill}} ll}

      {\bf Output:} \hspace{2mm} 
      belief ${\tt q} \in \Delta_0$ &

    \end{tabular*}
  \end{minipage}\end{figure} 
\end{multicols}

\vspace{-1.5em}
\begin{multicols}{2}

\begin{figure}[H] \smaller
  {\bf A.} ${\rm GBP}$ $\ee$-diffusion \\[0.5em]
\begin{minipage}{0.4\textwidth}
  \hrule

    \smaller
  \begin{algorithmic}[1] \label{alg-GBP}
    \setstretch{1.1}
    \Statex
    \For{$\tt i = 0 \dots n_{it}$} 
      \State {$\tt U_\aa \gets \zeta(u)_\aa$} 
      \State {$\tt U_\aa \gets U_\aa + \ln \Sigma \e^{-U_\aa}$} 
      \State {$\Phi_{\aa\bb} \tt\gets -\DF(U)_{\aa\bb}$} 
      \State {}
      \State {${\tt u_\aa \gets u_\aa } + \ee \cdot \delta(\Phi)_\aa$} 
    \EndFor
    \State {$\tt q_\aa \gets \e^{-U_\aa}$}\\
    \Return {${\tt q}$}
  \end{algorithmic}
\end{minipage}
\end{figure}

\begin{figure}[H] \smaller
  {\bf B.} Bethe $\ee$-diffusion\\[0.5em]
\begin{minipage}{0.4\textwidth}

  \hrule

    \smaller
  \begin{algorithmic}[1] \label{alg-Bethe}
    \setstretch{1.1}
    \Statex
    \For{$\tt i = 0 \dots n_{it}$} 
      \State {$\tt U_\aa \gets \zeta(u)_\aa$} 
      \State {}
      \State {$ \Phi_{\aa\bb} \tt \gets -\DF(U)_{\aa\bb}$} 
      \State {$ \phi_{\aa\bb} \gets {\tt c_\aa} \cdot \Phi_{\aa\bb}$} 
      \State {$\tt u_\aa \gets u_\aa + \ee \cdot \delta(\phi)_\aa$} 
    \EndFor
    \State {$\tt q_\aa \gets \e^{-U_\aa}$}\\
    \Return {$ {\tt q}$}
  \end{algorithmic}
\end{minipage}
\end{figure}
\end{multicols}
\vspace{-6mm}
\hr{0.4mm}
\end{figure}

\footnotetext{
  Note the normalisation operation 
  $U_\aa \gets U_\aa + \ln Z_\aa$ line 3 in A.
  It is replaced by line 4 in B, which takes care of 
  harmonising normalisation factors by eliminating
  redundancies in $\Phi$. 
  The arrows $U_{\aa} \gets \dots$ suggest $\tt map$ operations 
  that may be efficiently parallelised through asynchronous streams,
  by locality 
  of the associated operators $\zeta, \DF, \delta \dots$. 
  Each stream performs local operations over tensors in $A_\aa$, 
  whose dimensions depend on the 
  cardinality of local configuration spaces 
  $E_\aa = \prod_{i \in \aa} E_i$.
}

  Both algorithms consist of time-step $\ee$ 
  discrete Euler integrators of diffusion equations 
  of the form (\ref{eq-diffusion}), for two different flux functionals.
  Generalised belief propagation 
  (GBP) is usually expressed multiplicatively 
  for $\ee = 1$ in terms of beliefs
  $q_\aa = \frac 1 {Z_\aa} \e^{-U_\aa}$ and messages 
  $m_{\aa\bb} = \e^{- \ph_{\aa\bb}}$. 
  A choice of $\ee < 1$ would appear as an exponent 
  in the product of messages by this substitution.
  This is different from {\it damping} techniques \cite{Knoll-2017}
  and has not been previously considered to our knowledge. 
  
  Bethe numbers $c_\aa$ would also appear as exponents 
  of messages in the multiplicative formulation 
  of algorithm B. The combinatorial regularisation 
  offered by Bethe numbers stabilises divergent oscillations 
  in non-constant directions on hypergraphs, 
  improving convergence of GBP diffusion 
  at higher diffusivities. When $K$ is a graph, 
  the two algorithms are actually equivalent, so that 
  Bethe numbers only regularise normalisation 
  factors in the degree $\geq 2$ case. 

  \begin{figure}[h]
\begin{center}
\begin{minipage}{0.9\textwidth}
  \img{1}{TABLE.png}

  \smaller
  {\bf Fig 2.} 
  Convergence of GBP and Bethe diffusions 
  for different values of diffusivity $0 < \ee < 1$ and 
  energy scales on the 2-horn, depicted in figure 1. Both diffusions 
  almost surely diverge for diffusivities $\ee \geq 1$, 
  so that the usual GBP algorithm is not represented in this table. 
\end{minipage}
\end{center}
\end{figure}

  Figure 2 shows the results of experiments conducted 
  on the simplest hypergraph
  $K$ for which GBP does not surely converge to the unique 
  solution $q \in [u] \cap A_0^\Gamma$, the horn $\LL^2$ 
  depicted in figure 1. 
  Initial potentials $u \in A_0$ were normally sampled according 
  to $h_\aa(x_\aa) \sim \frac 1 T {\cal N}(0, 1)$ at 
  different temperatures or energy scales $T > 0$. 
  For each value of $T$ and for each fixed diffusivity 
  $\ee > 0$, GBP and Bethe diffusion 
  algorithms were run on random initial conditions
  for ${\tt n_{it}} = 10$ iterations. 
  Consistency of the returned beliefs, if any, was 
  assessed in the effective gradient $\Phi$ 
  to produce the represented decay ratios.
  Diffusivity was then increased until the drop in Bethe diffusion 
  convergence, occuring significantly later than GBP diffusion 
  but {\it before} $\ee < 1$, reflecting 
  the importance of using finer integrators than 
  usual $\ee = 1$ belief propagation algorithms. 

  The discretised diffusion 
  $(1 + \ee \delta \Phi)^n$ may be compared to 
  the approximate integration of
  $\exp(- n\ee x)$ as $(1 - \ee x)^n$, 
  which should only be done under the constraint $\ee |x| < 1$. 
  Assuming all eigenvalues of the linearised diffusion flow 
  $\delta \Phi_*$ are negative 
  (as is the case in the neighbourhood of a stable potential), 
  one should still ensure $\ee | \delta \Phi_* | < 1$ 
  to confidently estimate the large time asymptotics of diffusion 
  as $\exp(n \ee \delta \Phi) \simeq (1 + \ee \delta \Phi)^n$ 
  and reach $\Gamma$. 

  An open-sourced python implementation of the above algorithms, 
  with implementations of the (co)-chain complex 
  $A_\bullet(K)$ for arbitrary hypergraphs $K$, 
  Bethe numbers, Bethe entropy and 
  free energy functionals, and other operations for
  designing marginal estimation algorithms is on github at 
  \href{https://github.com/opeltre/topos}{opeltre/topos}.
  



\bibliographystyle{siam}
\bibliography{biblio}

\end{document}